\newif\iflibertine
\libertinetrue

\pdfoutput=1
\documentclass{article}
\usepackage{amsmath, amsthm, amssymb}
\usepackage{eucal}
\usepackage{mathrsfs}
\usepackage{geometry}
\usepackage[bottom]{footmisc}
\usepackage[unicode,bookmarksnumbered=true,colorlinks=true,allcolors=blue,linktoc=all]{hyperref}
\usepackage[capitalise,nameinlink]{cleveref}
\usepackage{tikz-cd}
\usepackage{stmaryrd}
\usepackage{listings}
\usepackage{etoolbox}
\usepackage{thm-restate}
\usepackage[nottoc,notlot,notlof]{tocbibind}
\usepackage{enumitem}
\usepackage{mathtools}
\usepackage{quiver}
\usepackage{caption}

\iflibertine
\usepackage[T1]{fontenc}
\usepackage{libertine}
\usepackage[libertine]{newtxmath}
\fi

\newlist{steps}{enumerate}{1}
\setlist[steps]{label=(\arabic*)}
\crefname{stepsi}{Step}{Steps}
\Crefname{stepsi}{Step}{Steps}

\crefname{subsection}{Subsection}{subsections}

\hypersetup{bookmarksdepth=2}
\newtheorem{theorem}{Theorem}

\newtheorem{thm}{Theorem}

\newtheorem{prop}[thm]{Proposition}

\theoremstyle{definition}

\newtheorem{notn}[thm]{Notation}

\newtheorem*{indhyp}{Inductive Hypothesis}

\newtheorem{remark}[thm]{Remark}

\newtheorem*{thm*}{Theorem}

\newcommand{\ncmd}{\newcommand}

\definecolor{DefColor}{rgb}{0.6,0.15,0.25}

\ncmd{\mathbfsf}[1]{\mathord{\text{\normalfont\bfseries #1}}}

\iflibertine
\ncmd{\mbb}[1]{\mathbfsf{#1}}
\else
\ncmd{\mbb}[1]{\mathbb{#1}}
\fi
\ncmd{\mrm}[1]{\mathrm{#1}}
\ncmd{\mcl}[1]{\mathcal{#1}}
\ncmd{\mfk}[1]{\mathfrak{#1}}
\ncmd{\mbf}[1]{\mathbf{#1}}
\ncmd{\mscr}[1]{\mathscr{#1}}

\ncmd{\todo}[1]{\textbf{TODO #1}}
\ncmd{\reftodo}[1]{\textbf{REF #1}}

\DeclareRobustCommand{\minwidthbox}[2]{%
  \mathmakebox[\ifdim#2<\width\width\else#2\fi]{#1}%
}
\ncmd{\too}[1][]{\xrightarrow{\minwidthbox{#1}{1em}}}
\ncmd{\oot}[1][]{\xleftarrow{\minwidthbox{#1}{1em}}}
\def\isoraise{\iflibertine -0.8ex\else -0.5ex\fi}
\ncmd{\iso}{\too[\smash{\raisebox{\isoraise}{\ensuremath{\scriptstyle\sim}}}]}
\ncmd{\osi}{\oot[\smash{\raisebox{\isoraise}{\ensuremath{\scriptstyle\sim}}}]}
\ncmd{\hooktoo}[1][]{\xhookrightarrow{\minwidthbox{#1}{1em}}}
\ncmd{\adj}[1][]{\mathrel{\substack{\xrightarrow{\minwidthbox{#1}{1em}} \\[-.7ex] \xleftarrow{\minwidthbox{#1}{1em}}}}}

\ncmd{\qin}{\quad\in\quad}

\ncmd{\Id}{\mrm{Id}}
\ncmd{\Nm}{\mathrm{Nm}}

\ncmd{\BB}{\mrm{B}}
\ncmd{\clB}{\mcl{B}}
\ncmd{\CC}{\mcl{C}}
\ncmd{\DD}{\mcl{D}}
\ncmd{\EE}{\mcl{E}}
\ncmd{\TT}{\mcl{T}}
\ncmd{\MM}{\mcl{M}}
\ncmd{\KK}{\mrm{K}}
\ncmd{\ku}{\mrm{ku}}
\ncmd{\KU}{\mrm{KU}}
\ncmd{\cX}{\mcl{X}}
\ncmd{\cY}{\mcl{Y}}
\ncmd{\Kn}{\KK(n)}
\ncmd{\Knp}{\KK(n{+}1)}
\ncmd{\Ko}{\KK(1)}
\ncmd{\Tn}{\mrm{T}(n)}
\ncmd{\Tnp}{\mrm{T}(n{+}1)}
\ncmd{\To}{\mrm{T}(1)}
\ncmd{\Tm}{\mrm{T}(m)}
\ncmd{\KTnp}{\KK_{\Tnp}}
\ncmd{\KKo}{\KK_{\Ko}}
\ncmd{\KTo}{\KK_{\To}}
\ncmd{\THH}{\mrm{THH}}
\ncmd{\TC}{\mrm{TC}}
\ncmd{\TCm}{\TC^-}
\ncmd{\An}{A_n}
\ncmd{\rmE}{\mrm{E}}
\ncmd{\En}{\rmE_n}
\ncmd{\Enp}{\rmE_{n{+}1}}
\ncmd{\BPn}{\mrm{BP}\langle n\rangle}
\ncmd{\JWnp}{\rmE(n{+}1)}
\ncmd{\cJWnp}{\widehat{\JWnp}}
\ncmd{\Enpk}{\Enp(\kappa)}
\ncmd{\lsF}{\mscr{F}}
\ncmd{\lsG}{\mscr{G}}
\ncmd{\bbC}{\mbb{C}}
\ncmd{\GG}{\mbb{G}}
\ncmd{\NN}{\mbb{N}}
\ncmd{\ZZ}{\mbb{Z}}
\ncmd{\QQ}{\mbb{Q}}
\ncmd{\Qbar}{\overline{\QQ}}
\ncmd{\clQ}{\mcl{Q}}
\ncmd{\Qab}{\QQ(\zeta_\infty)}
\ncmd{\Zp}{\ZZ_p}
\ncmd{\Qp}{\QQ_p}
\ncmd{\Fp}{\mbb{F}_p}
\ncmd{\Fpbar}{\overline{\mbb{F}}_p}
\ncmd{\Fell}{\mbb{F}_\ell}
\renewcommand{\SS}{\mbb{S}}
\ncmd{\SKn}{\SS_{\Kn}}
\ncmd{\SKnp}{\SS_{\Knp}}
\ncmd{\SKo}{\SS_{\Ko}}
\ncmd{\STn}{\SS_{\Tn}}
\ncmd{\STnp}{\SS_{\Tnp}}
\ncmd{\rmP}{\mrm{P}}
\ncmd{\OO}{\mcl{O}}
\ncmd{\one}{\mbf{1}}
\ncmd{\bbE}{\mbb{E}}

\ncmd{\irchi}[2]{\raisebox{\depth/2}{$#1\chi$}}
\DeclareRobustCommand{\rchi}{{\mathpalette\irchi\relax}}
\ncmd{\ch}{\rchi}
\ncmd{\cch}{\widehat{\scalebox{1.15}{$\rchi$}}}

\ncmd{\otimesu}{\mathop{\otimes}\limits}

\ncmd{\fL}{L}
\ncmd{\fLLam}{\fL^\Lambda}
\ncmd{\LKn}{L_{\Kn}}
\ncmd{\LKnp}{L_{\Knp}}
\ncmd{\LKo}{L_{\Ko}}
\ncmd{\LTn}{L_{\Tn}}
\ncmd{\LTnp}{L_{\Tnp}}
\ncmd{\LTm}{L_{\Tm}}
\ncmd{\Ln}{L_n}
\ncmd{\Lnm}{L_{n{-}1}}
\ncmd{\Mn}{M_n}
\ncmd{\Mnf}{M_n^f}
\ncmd{\Lof}{L_1^f}
\ncmd{\Lnf}{L_n^f}
\ncmd{\Lnmf}{L_{n{-}1}^f}
\ncmd{\Lnpf}{L_{n{+}1}^f}
\ncmd{\Mz}{M_0}
\ncmd{\Lz}{L_0}

\ncmd{\LL}{\mrm{L}}
\ncmd{\RR}{\mrm{R}}
\ncmd{\BC}{\mrm{BC}}
\ncmd{\psa}{\oplus}
\ncmd{\dbl}{\mrm{dbl}}
\ncmd{\op}{\mrm{op}}
\ncmd{\pifin}{\pi\text{-}\mathrm{fin}}
\ncmd{\pfin}{p\text{-}\mathrm{fin}}
\ncmd{\seg}{\mrm{seg}}
\ncmd{\perf}{\mrm{perf}}
\ncmd{\st}{\mrm{st}}

\ncmd{\pt}{\mrm{pt}}
\ncmd{\Perf}{\mrm{Perf}}
\ncmd{\Mod}{\mrm{Mod}}
\ncmd{\cMod}{\widehat{\Mod}\vphantom{\Mod}}
\ncmd{\cModdbl}{\cMod^{\dbl}}
\ncmd{\Vect}{\mrm{Vect}}
\ncmd{\Ab}{\mrm{Ab}}
\ncmd{\Fin}{\mrm{Fin}}
\ncmd{\Spaces}{\mcl{S}}
\ncmd{\Spacespifin}{\Spaces_{\pifin}}
\ncmd{\Spacespfin}{\Spaces_{\pfin}}
\ncmd{\Sp}{\mrm{Sp}}
\ncmd{\SpTn}{\Sp_{\Tn}}
\ncmd{\SpTnp}{\Sp_{\Tnp}}
\ncmd{\SpTm}{\Sp_{\Tm}}
\ncmd{\SpKn}{\Sp_{\Kn}}
\ncmd{\SpKo}{\Sp_{\Ko}}
\ncmd{\SpKnp}{\Sp_{\Knp}}
\ncmd{\Span}{\mrm{Span}}
\ncmd{\Spano}{\Span_1}
\ncmd{\equivSp}{\underline{\Sp}}
\ncmd{\Cat}{\mrm{Cat}}
\ncmd{\Catpifin}{\Cat_{\pifin}}
\ncmd{\Catpfin}{\Cat_{\pfin}}
\ncmd{\CatLn}{\Cat_{\Ln}}
\ncmd{\CatLnm}{\Cat_{\Lnm}}
\ncmd{\CatMn}{\Cat_{\Mn}}
\ncmd{\CatMnf}{\Cat_{\Mnf}}
\ncmd{\CatLnf}{\Cat_{\Lnf}}
\ncmd{\Catperf}{\Cat_{\perf}}
\ncmd{\Catst}{\Cat_{\st}}
\ncmd{\PrL}{\mrm{Pr}^\mrm{L}}
\ncmd{\PrLst}{\PrL_{\mrm{st}}}
\ncmd{\PrLstw}{\PrL_{\mrm{st},\omega}}
\ncmd{\PrLKn}{\PrL_{\Kn}}
\ncmd{\PrLKnw}{\PrL_{\Kn,\omega}}
\ncmd{\Mfg}{\mcl{M}_\mrm{fg}}
\ncmd{\Thick}{\mrm{Thick}}

\ncmd{\yon}{\text{\usefont{U}{min}{m}{n}\symbol{'110}}}
\DeclareFontFamily{U}{min}{}
\DeclareFontShape{U}{min}{m}{n}{<-> dmjhira}{}

\DeclareMathOperator{\Alg}{Alg}
\DeclareMathOperator{\CAlg}{CAlg}
\ncmd{\CAlgG}{\CAlg_G}
\DeclareMathOperator{\CMon}{CMon}
\DeclareMathOperator{\coCMon}{coCMon}
\ncmd{\CMoninf}{\CMon_\infty}
\ncmd{\coCMoninf}{\coCMon_\infty}
\DeclareMathOperator*{\colim}{colim}

\ncmd\noloc{%
  \nobreak
  \mspace{6mu plus 1mu}
  {:}
  \nonscript\mkern-\thinmuskip
  \mathpunct{}
  \mspace{2mu}
}

\title{\vspace{-1.2em}The Redshift Bound from Quillen--Lichtenbaum}
\author{Shay Ben-Moshe\thanks{Max Planck Institute for Mathematics, Bonn, Germany.} \thanks{Faculty of Mathematics and Computer Science, Weizmann Institute of Science, Israel.}}
\date{}

\begin{document}
	\maketitle

	\vspace{-1.8em}
	
	\begin{abstract}
		We give a new proof that algebraic K-theory increases chromatic height by at most one, originally established by Clausen--Mathew--Naumann--Noel. Our argument proceeds by descent from the Lubin--Tate spectrum, for which the required vanishing follows from Hahn--Wilson's Quillen--Lichtenbaum result.
	\end{abstract}

	
	
	\section*{Introduction}

Two of the foundational results that have uncovered the relationship between algebraic K-theory and the chromatic filtration are Mitchell's theorem \cite{mitchell} and the Quillen--Lichtenbaum conjecture proved by Voevodsky--Rost \cite{Voevodsky1,Voevodsky2}.
The first says that the algebraic K-theory of any ordinary ring $R$ vanishes at chromatic heights above $1$, namely
\[
	\LTm \KK(R) = 0
	\quad \text{for all $m \geq 2$}.
\]
The second, as reinterpreted by Waldhausen \cite{Waldhausen1984}, says that for suitable rings $R$, the map
\[
	\KK(R)_{(p)} \too \Lof \KK(R),
\]
which approximates K-theory by chromatic heights $0$ and $1$, has bounded above fiber.
We note that the latter implies the former for these suitable rings, since $\Tm$-localization vanishes on bounded above spectra.

Based on these and computational evidence from height $1$, Ausoni--Rognes have outlined the far-reaching redshift conjectures, generalizing these phenomena to higher chromatic heights \cite{RognesOberwolfach,AR02,AR08}.
These conjectures have seen remarkable progress in recent years.
In particular, the higher height analogue of Mitchell's theorem was proved by Clausen--Mathew--Naumann--Noel \cite{DescVan}, using the purity theorem of Land--Mathew--Meier--Tamme \cite{purity}, giving the following redshift upper bound.

\begin{theorem}[{\cite[Theorem C]{DescVan}, \cref{main-thm}}]\label{main-thm-intro}
	If $\CC \in \Catperf$ is $\Lnf$-local then
	\[
		\LTm \KK(\CC) = 0
		\quad \text{for all $m \geq n{+}2$}.
	\]
\end{theorem}

Almost concurrently, Hahn--Wilson proved a higher height analogue of the Quillen--Lichtenbaum conjecture \cite{HW}, later revisited by Angelini-Knoll \cite{AK-BPn}.
They constructed a certain $\bbE_3$-form of the truncated Brown--Peterson spectrum $\BPn$ for which they proved the following.

\begin{thm}[{\cite[Theorem B]{HW}}]\label{QL}
	The following map has a bounded above fiber
	\[
		\KK(\BPn)_{(p)} \too \Lnpf \KK(\BPn).
	\]
\end{thm}

Yuan proved the redshift lower bound for the Lubin--Tate spectrum $\En$ \cite{yuan}.
Burklund--Schlank--Yuan subsequently extended this to all commutative ring spectra using their chromatic nullstellensatz, which expresses one sense in which Lubin--Tate spectra behave like algebraically closed fields \cite{null}.

In this paper, we give a new proof of \cref{main-thm-intro} by descent from $\En$.
The vanishing for the Lubin--Tate spectrum follows from the Quillen--Lichtenbaum result of \cref{QL}, and the passage to the general case relies on the following more familiar sense in which $\En$ is algebraically closed.
The pioneering work of Devinatz--Hopkins \cite{DH}, together with subsequent work of Rognes, Baker--Richter, and Mathew \cite{RognesGalois,BR-Galois,MatGal}, shows that $\En$ is the maximal Galois extension of the $\Kn$-local sphere $\SKn$.
Recent unpublished work of Burklund--Clausen--Levy shows that $\STn$ has the same Galois theory as $\SKn$, and that its algebraic closure coincides with $\En$ as well, which in particular gives the following.

\begin{thm}[{Burklund--Clausen--Levy}]\label{BCL}
	The $\Tn$-local filtered colimit of the $\Tn$-local finite Galois extensions $R$ of $\STn$ is $\En$, that is, there is an isomorphism
	\[
		\LTn(\colim_{R/\STn} R) \iso \En
		\qin \CAlg(\SpTn).
	\]
\end{thm}

\begin{remark}
	Our proof relies on this unpublished result.
	Using instead the $\Kn$-local algebraic closure statement, the argument applies to $\Ln$-local categories.
\end{remark}

We briefly outline the proof.
We start from the vacuous case $n = -1$, where we take $\Lnf$ to be the zero functor, and proceed by induction on the height $n$ as follows:

\begin{steps}
	\item \cref{QL} implies the vanishing for $\En$ via the map from $\BPn$ (\cref{QL-En}).
	\item Since K-theory preserves filtered colimits, and a ring spectrum vanishes if and only if its unit vanishes, \cref{BCL} then gives the vanishing for some finite Galois extension of $\STn$ (\cref{some-vanishing}).
	\item We deduce the vanishing for $\STn$ using a general vanishing descent principle, which may be of independent interest (\cref{descent}, \cref{SpTn-dbl}).
	\item The vanishing for $\STn$ together with the vanishing for $\Lnmf\SS$ from the inductive hypothesis implies the vanishing for $\Lnf\SS$ using the telescopic fracture square (\cref{fracture}, \cref{LnfS}).
	\item An $\Lnf$-local category is a module over $\Perf(\Lnf\SS)$, hence the same is true for their K-theories, so the vanishing of the latter implies the vanishing of the former (\cref{main-thm}).
\end{steps}

A subtlety suppressed in this outline is the distinction between ordinary and $\Tn$-local rings and their module categories.
The telescopic fracture square, together with the inductive hypothesis, shows that this discrepancy is invisible to K-theory localized at the relevant chromatic heights (\cref{fracture}).

\subsection*{Acknowledgements}

The proof in this paper was inspired by Elmanto--Nardin--Yang's proof of Mitchell's theorem \cite{ENY}.
I thank Shachar Carmeli and Ishan Levy for pointing out an oversight in the proof of \cref{some-vanishing} appearing in an earlier version of this paper.
I also thank John Rognes and Tomer Schlank for helpful exchanges.
This work was supported by the Max Planck--Weizmann joint postdoctoral program.

I used ChatGPT 5.6 throughout this project, including for locating references and streamlining the exposition, and, notably, for suggesting the proof of \cref{descent}.

	\section*{The Proof}

\begin{indhyp}\label{inductive-assumption}
	We assume that \cref{main-thm-intro} has been proved for height $n-1$.
\end{indhyp}

\begin{notn}
	For a $\Tn$-local ring spectrum $R$, we denote by $\cMod_R$ the category of $\Tn$-local $R$-modules.
\end{notn}

We begin by recording the following statement, which we shall use twice.

\begin{prop}\label{fracture}
	For any ring spectrum $R$, the following is a pullback square
	\[\begin{tikzcd}
		{\KK(\Lnf R)} & {\KK(\LTn R)} \\
		{\KK(\Lnmf R)} & {\KK(\Lnmf\LTn R)}
		\arrow[from=1-1, to=1-2]
		\arrow[from=1-1, to=2-1]
		\arrow["\lrcorner"{anchor=center, pos=0.125}, draw=none, from=1-1, to=2-2]
		\arrow[from=1-2, to=2-2]
		\arrow[from=2-1, to=2-2]
	\end{tikzcd}\]
	Consequently, for any $m \geq n{+}2$ the top map induces an isomorphism
	\[
		\LTm \KK(\Lnf R) \iso \LTm \KK(\LTn R).
	\]
\end{prop}

\begin{proof}
	As in the proof of \cite[Corollary 4.11]{DescVan}, this follows from the fact that both vertical fibers are the K-theory of the $n$-monochromatization of $\Perf(R)$.
	We give another, closely related proof.

	The telescopic fracture says that the following is a pullback square of ring spectra
	\[\begin{tikzcd}
		{\Lnf R} & {\LTn R} \\
		{\Lnmf R} & {\Lnmf\LTn R}
		\arrow[from=1-1, to=1-2]
		\arrow[from=1-1, to=2-1]
		\arrow["\lrcorner"{anchor=center, pos=0.125}, draw=none, from=1-1, to=2-2]
		\arrow[from=1-2, to=2-2]
		\arrow[from=2-1, to=2-2]
	\end{tikzcd}\]
	By \cite[Corollary 1.4]{LT}, it then remains to show that the map of spectra
	\[
		\Lnmf R \otimesu_{\Lnf R} \LTn R \too \Lnmf\LTn R
	\]
	is an isomorphism.
	This follows from the fact that $\Lnmf$ is a smashing localization
	\[
		\Lnmf R \otimesu_{\Lnf R} \LTn R
		\simeq (\Lnmf\SS \otimes \Lnf R) \otimesu_{\Lnf R} \LTn R
		\simeq \Lnmf\SS \otimes (\Lnf R \otimesu_{\Lnf R} \LTn R)
		\simeq \Lnmf\SS \otimes \LTn R
		\simeq \Lnmf\LTn R.
	\]
	For the consequence, note that by the \hyperref[inductive-assumption]{Inductive Hypothesis}, the two bottom terms vanish $\Tm$-locally.
\end{proof}

We now turn our attention to the proof, proceeding by a series of reductions.

\begin{prop}\label{QL-En}
	$\LTm \KK(\En) = 0$ for all $m \geq n{+}2$.
\end{prop}

\begin{proof}
	By \cite[Proposition 8.3]{ABM}, there is a map
	\[
		\BPn \to E
		\qin \Alg_{\bbE_3}(\Sp)
	\]
	where $\BPn$ is the Hahn--Wilson form of $\BPn$, and $E$ is some $\bbE_3$-ring spectrum whose underlying $\bbE_1$-ring spectrum is $\En$.
	As a consequence, we have
	\[
		\LTm \KK(\BPn) \too \LTm \KK(E) \iso \LTm \KK(\En)
	\]
	where the first map is an $\bbE_2$-ring map, and the second map is an isomorphism of spectra.
	\cref{QL} shows that the first term vanishes since $\Tm$-localization vanishes on bounded above spectra for $m \geq 1$.
	Since the first map is a ring map, the second term vanishes as well, hence so does the third term.
\end{proof}

As recalled in the introduction, \cref{BCL} of Burklund--Clausen--Levy shows that $\En$ is the $\Tn$-localization of the filtered colimit \emph{in spectra} of the $\Tn$-local finite Galois extensions $R$ of $\STn$
\[
	\An := \colim_{R/\STn} R
	\qin \CAlg(\Sp),
	\qquad \LTn \An \iso \En.
\]

\begin{prop}\label{vanishing-An}
	$\LTm \KK(\An) = 0$ for all $m \geq n{+}2$.
\end{prop}

\begin{proof}
	Note that $\An$ is $\Lnf$-local since $\Lnf$-localization is smashing, and we have $\LTn\An \iso \En$.
	Therefore, \cref{fracture} shows that $\An \to \En$ induces an isomorphism on $\Tm$-localized K-theory, so the result follows from \cref{QL-En}.
\end{proof}

Next, we pass from $\An$ to a finite Galois extension of $\STn$.
We learned this argument from Elmanto--Nardin--Yang's proof of Mitchell's theorem \cite{ENY}.

\begin{prop}\label{some-vanishing}
	For any $m \geq n{+}2$, there is a $\Tn$-local finite $G$-Galois extension $R$ of $\STn$, such that $\LTm \KK(R) = 0$, hence also $\LTm \KK(\cModdbl_R) = 0$.
\end{prop}

\begin{proof}
	Using \cref{vanishing-An}, together with the fact that K-theory preserves filtered colimits, we get
	\[
		\LTm(\colim_{R/\STn} \KK(R)) \iso \LTm\KK(\An) = 0
		\qin \CAlg(\SpTm).
	\]
	Choose $\Tm$ so that it is a ring spectrum.\footnote{This can be done by taking the telescope on the endomorphism ring of a finite type $m$ spectrum. See for example \cite[Proof of Lemma 2.3]{purity} for a more detailed argument.}
	Since having vanishing $\Tm$-localization is equivalent to being $\Tm$-acyclic, we get
	\[
		\colim_{R/\STn} (\Tm \otimes \KK(R)) \simeq \Tm \otimes (\colim_{R/\STn} \KK(R)) = 0
		\qin \Alg(\Sp).
	\]
	Since the colimit is along ring spectra maps, it preserves the unit element.
	Recall that a ring spectrum vanishes if and only if its unit vanishes.
	Therefore, since the sphere spectrum is compact, we conclude that there is some $R/\STn$ such that the unit element vanishes in $\Tm \otimes \KK(R)$, or, equivalently, $\LTm \KK(R) = 0$.

	The second part follows from the fact that the functor
	\[
		\Perf(R) \too \cModdbl_R
	\]
	is symmetric monoidal, hence induces a map of commutative ring spectra
	\[
		\LTm \KK(R) \too \LTm \KK(\cModdbl_R),
	\]
	so vanishing of the source implies vanishing of the target.
\end{proof}

In the next step we will use the following general vanishing descent principle.

\begin{prop}\label{descent}
	Let $G$ be a finite group, let $\CC \in \CAlg(\Catperf)^{\BB G}$, and let $Z$ be a spectrum, then
	\[
		L_Z \KK(\CC) = 0
		\qquad\text{if and only if}\qquad L_Z \KK(\CC^{hG}) = 0.
	\]
\end{prop}

\begin{proof}
	For the if direction, the symmetric monoidal forgetful functor $\CC^{hG} \to \CC$ induces a map of commutative ring spectra on $Z$-localized K-theory
	\[
		L_Z \KK(\CC^{hG}) \too L_Z \KK(\CC),
	\]
	so the vanishing of the source implies the vanishing of the target.

	For the only if direction, recall that equivariant K-theory assembles into a normed $G$-ring spectrum by \cite[Example 3.3.4]{HR} (see also \cite[Corollary 4.3.9]{EH} and \cite[Corollary C]{CHLL})
	\[
		\KK_G(\CC) \qin \CAlgG(\equivSp),
		\qquad \KK_G(\CC)^H \simeq \KK(\CC^{hH}) \qin \CAlg(\Sp).
	\]
	Applying $Z$-localization level-wise at each categorical fixed point, we get a normed $G$-ring spectrum (see \cite[Theorem 3.23]{Hill}, noting that $L_{i_* Z}$ in \emph{loc.\ cit.} is equivalent to $Z$-localization level-wise since they have the same acyclics by \cite[Proposition 3.20]{Hill})
	\[
		L_Z \KK_G(\CC) \qin \CAlgG(\equivSp),
		\qquad (L_Z \KK_G(\CC))^H \simeq L_Z \KK(\CC^{hH}) \qin \CAlg(\Sp).
	\]
	Consider the counit of the norm-restriction adjunction, giving a map
	\[
		N_e^G(L_Z \KK(\CC)) \simeq N_e^G((L_Z \KK_G(\CC))^e) \too L_Z \KK_G(\CC)
		\qin \CAlgG(\equivSp).
	\]
	Taking categorical $G$-fixed points, we get a map of commutative ring spectra
	\[
		(N_e^G(L_Z \KK(\CC)))^G \too L_Z \KK(\CC^{hG})
		\qin \CAlg(\Sp).
	\]
	By assumption the source vanishes, and since this is a ring map, the target vanishes as well.
\end{proof}

With this in place, we return to the main line of the proof.

\begin{prop}\label{SpTn-dbl}
	$\LTm \KK(\SpTn^\dbl) = 0$ for all $m \geq n{+}2$.
\end{prop}

\begin{proof}
	Fix $m \geq n{+}2$, and choose a $\Tn$-local finite $G$-Galois extension $R/\STn$ as in \cref{some-vanishing}.
	We observe that \cite[Proposition 6.3.3]{RognesGalois} shows that $R$ is faithful over $\STn$ in light of the $1$-semiadditivity of $\Tn$-local spectra \cite{Kuhn}.
	Hence, we have a symmetric monoidal equivalence (see for example \cite[Proposition 3.11]{Desc})
	\[
		\SpTn \iso (\cMod_R)^{hG},
	\]
	so upon passing to dualizable objects the result follows from \cref{descent}.
\end{proof}

\begin{prop}\label{STn}
	$\LTm \KK(\STn) = 0$ for all $m \geq n{+}2$.
\end{prop}

\begin{proof}
	Recall from \cite[Proposition 4.15]{DescVan} that
	\[
		\Perf(\STn) \too \SpTn^\dbl
	\]
	is fully faithful and its Verdier quotient $\DD$ is $\Lnmf$-local.
	Applying $\Tm$-localized K-theory, we get an exact sequence
	\[
		\LTm \KK(\STn) \too \LTm \KK(\SpTn^\dbl) \too \LTm \KK(\DD).
	\]
	By \cref{SpTn-dbl} the middle term vanishes and by the \hyperref[inductive-assumption]{Inductive Hypothesis} the last term vanishes, hence so does the first term.
\end{proof}

\begin{prop}\label{LnfS}
	$\LTm \KK(\Lnf\SS) = 0$ for all $m \geq n{+}2$.
\end{prop}

\begin{proof}
	By \cref{fracture}, the map $\Lnf\SS \to \STn$ induces an isomorphism on $\Tm$-localized K-theory, so the result follows from \cref{STn}.
\end{proof}

Finally, we restate and prove the main theorem.

\begin{thm}[{\cref{main-thm-intro}}]\label{main-thm}
	If $\CC \in \Catperf$ is $\Lnf$-local then $\LTm \KK(\CC) = 0$ for all $m \geq n{+}2$.
\end{thm}

\begin{proof}
	Any $\Lnf$-local $\CC \in \Catperf$ is a module over $\Perf(\Lnf\SS)$, hence $\LTm \KK(\CC)$ is a module over $\LTm \KK(\Lnf\SS)$.
	The latter vanishes by \cref{LnfS}, hence so does the former.
\end{proof}

	\bibliographystyle{alpha}
	\bibliography{refs}

\end{document}